\documentclass[12pt]{amsart}
\usepackage{amsmath,amssymb, amsfonts,amscd,psfrag,graphicx, enumitem}
\usepackage{mathrsfs}
\newcommand{\lebn}

\theoremstyle{plain}

\newtheorem{thm}[equation]{Theorem}

\newtheorem*{thm1}{Theorem \ref{thm1}}

\newtheorem{cor}[equation]{Corollary}
\newtheorem{lem}[equation]{Lemma}

\theoremstyle{definition}

\newtheorem{defn}[equation]{Definition}
\newtheorem{rem}[equation]{Remark}
\newtheorem{exmp}[equation]{Example}
  
\numberwithin{equation}{section}

\newcommand{\ME}{\mathcal{M}}

\newcommand{\DE}{\mathcal{D}}
\newcommand{\FE}{\mathcal{F}}

\newcommand{\IE}{\operatorname{Ind}}

\newcommand{\KE}{\mathcal{K}}

\newcommand{\ind}{\operatorname{ind}}
\newcommand{\m}{\operatorname{m}}

\newcommand{\lk}{\operatorname{link}}
\newcommand{\dl}{\operatorname{del}}
\newcommand{\ve}{\operatorname{v}}
\newcommand{\Dom}{\operatorname{Dom}}
\newcommand{\D}{\Delta}

\newcommand{\iso}{\cong}

\begin{document}

\bibliographystyle{plain}

\title[Matching Trees and Homotopy Type Of Devoid Complexes Of Graphs]{Matching Trees for Simplicial Complexes and Homotopy Type Of Devoid Complexes Of Graphs}
\author{Demet Taylan}
\address{Department of Mathematics, University Of Bremen, 28359 Bremen, Germany}
\email{dtaylan@informatik.uni-bremen.de}

\address{Department of Mathematics, Bozok University, Yozgat, 66900, Turkey.}
\email{demet.taylan@bozok.edu.tr}


\date{\today}

\thanks{The author is supported by T\" UB\. ITAK,
grant no: 111T704.}


\begin{abstract} 
We generalize some homotopy calculation techniques such as splittings and matching trees that are introduced for the 
computations in the case of the independence complexes of graphs to arbitrary simplicial complexes, and exemplify their efficiency on some
simplicial complexes, the devoid complexes of graphs, $\DE(G;\FE)$ whose faces are vertex subsets of $G$ that induce $\FE$-free subgraphs,
where $G$ is a multigraph and $\FE$ is a family of multigraphs. Additionally, we compute the homotopy type of dominance complexes of chordal graphs.
\end{abstract}

\maketitle

\section{Introduction}
In recent years the efficiency of topological methods on solving combinatorial problems has been demonstrated in various papers,
whose starting point lies in Lov\'asz's proof of the Kneser's conjecture.
In this guise, determining the homotopy types of simplicial complexes plays a crucial role. When a simplicial complex $\D$ is flag, that is,
if it is the independence complex $\D=\IE(G)$ of a graph $G$, the underlying combinatorial structure of $G$ can provide enough information that may ease
the calculation of the homotopy type. For instance, there exists a reduction technique that splits the independence complex homotopically 
into a wedge of smaller ones, namely that $\IE(G)\simeq \IE(G-x)\vee \Sigma(\IE(G-N_G[x])$ for any vertex $x$ of the graph that possesses a neighborhood
$y$ satisfying $N_G[y]\subseteq N_G[x]$ in $G$~\cite{A,MT}. Similarly, in the case of the independence complexes,
Bousquet-M\'elou, et al. in \cite{BLN} have introduced the notion of a matching tree
in order to construct a Morse matching for $\IE(G)$ and used it to compute the homotopy types of independence complexes of some grid graphs.  

Our primary aim here is to generalize these techniques that may be in help to compute the homotopy types of non-flag simplicial complexes as well.  
For example, for an arbitrary simplicial 
complex $\D$ and a subset $A\subseteq V(\D)$, if we define $\DE(A):=\{F\in \D\colon F\cup A\notin \D\}$, then we have the homotopy equivalence 
$\D\simeq \dl_{\D}(p)$ provided that there exists a vertex $q$ other than $p$ satisfying $\DE(\{q\})\subseteq \DE(\{p\})$. Similarly, we introduce the
notion of matching trees for an arbitrary simplicial complex by considering the sets $\DE(A)$ that enables us to construct a Morse matching for such a complex.

We exhibit the effectiveness of such generalizations by computing the homotopy types of certain non-flag simplicial complexes.
The complexes that we consider are again parametrized by graphs. To be more specific, if $G$ is a (multi)graph and $\FE$ is a family of
(multi)graphs, we then define the \emph{devoid complex} $\DE(G;\FE)$ to be the simplicial complex on $V(G)$ whose faces are those
subsets $D\subseteq V(G)$ such that $G[D]$ is $\FE$-free (see Section~\ref{section:reduct} for details). As an application, we compute the homotopy type of
$\DE(C_n;P_k)$ as follows:

\begin{thm1} For the devoid complex $\DE(C_n;P_k)$, the following homotopy equivalence holds:
\begin{equation*}
\DE(C_n;P_k)\simeq
\begin{cases}
\bigvee^k S^{t(k-1)-1}, & \text{if $n=(k+1)t$},\\
S^{t(k-1)-1}, & \text{if $n=(k+1)t+1$},\\
S^{t(k-1)+d-2}, & \text{if $n=(k+1)t+d$},\\
S^{t(k-1)+k-2}, & \text{if $n=(k+1)t+k$},\\
\end{cases}
\end{equation*}
where $2\leq d\leq k-1$.
\end{thm1}   

We further compute the homotopy types of $\DE(P_n;P_k)$ (Theorem~\ref{thm:paths}) and $\DE(F;P_3)$ (Theorem~\ref{thm:forest}), where $F$ is a forest.

The paper is organized as follows. In Section $2$ we recall some basic notions about graphs and simplicial complexes, and collect some necessary topological background. 
We review the basics of the discrete Morse theory, and recall the construction of Morse matchings via \emph{matching trees} for independence complexes of graphs which 
was introduced in \cite{BLN}. In Section $3$ we describe some homotopy reduction techniques for arbitrary simplicial complexes, and apply them to calculate the homotopy 
type of some devoid complexes. In Section $4$ we construct a \emph{matching tree} as a tool to find Morse matchings for arbitrary simplicial complexes, and 
apply this machinery to determine the homotopy type of some particular simplicial complexes in Section $5$.

\section{Preliminaries}\label{pre}
In this section we recall some general notions and exhibit the tools from discrete Morse theory.
\subsection{Graphs}
By a graph $G$ we mean an undirected (multi)graph. If $G$ is a graph, $V(G)$ and $E(G)$ (or simply $V$ and $E$) denote its vertex and edge sets. An edge between $u$ and $v$ is denoted by $e=(u,v)$. A forest is a cycle-free graph, while a tree is a connected forest. If $U\subseteq V$ then $G\setminus U$ is the graph induced on the vertex set $V\setminus U$. We abbreviate $G\setminus \{x\}$ to $G\setminus x$. The subgraph of $G$ that is induced by $U$ will be denoted by $G[U]$. The degree of a vertex $x$ in $G$ will be denoted by $d_G(x)$. A vertex $x$ of $G$ is called \emph{discrete} if $d_G(x)=0$. We denote the set of neighbors of a vertex $x$ of $G$ by $N_G(x)$ (or $N(x)$). The closed neighborhood of a vertex $x$ of $G$ is denoted by $N_G[x]$ (or $N[x]$) and $N[x]=N(x)\cup \{x\}$. If two graphs $G$ and $H$ are isomorphic, we denote it by $G\iso H$. A graph is called $F$-free if it contains no subgraph which is isomorphic to $F$,
and if $\FE$ is a family of graphs, then a graph $G$ is said to be $\FE$-free, if $G$ is $F$-free for each $F\in \FE$.

A subset $I$ of the vertex set $V$ of $G$ is called \emph{independent} if no two vertices of $I$ are adjacent. A matching of $G$ is a set of pairwise disjoint edges. Maximum size of a \emph{matching} of $G$ is called the \emph{matching number} of $G$ and denoted by $\m(G)$. A subset $D$ of $V$ is said to be a \emph{dominating set} in $G$ if every vertex not in $D$ is adjacent to at least one vertex of $D$. A \emph{vertex cover} of $G$ is a subset $C\subseteq V$ such that every edge of $G$ contains a vertex of $C$. The \emph{vertex covering number} of a graph $G$ is the minimum size of a vertex cover of $G$ and it is denoted by $\ve(G)$.

Throughout $K_n$, $C_n$, $P_n$ will denote the complete, cycle and path graphs on $n$ vertices, respectively. Also $K_{m,n}$ denotes the complete bipartite graph with partitions of size $m$ and $n$. In particular, a $2$-cycle $C_2$ corresponds to a double-edge.

The following definition seems to first appear in \cite{AHK}.
\begin{defn}Let $F$ and $G$ be graphs. An $F$-\emph{matching} in $G$ is a set of pairwise vertex disjoint copies of $F$. 
An \emph{induced} $F$-matching in $G$ is an $F$-matching such that no additional edge of $G$ is spanned by the vertices 
of $G$ covered by the matching. Note that a matching is a $K_2$-matching. We will denote the maximum size of an induced 
$F$-matching of a graph $G$ by $\ind_F(G)$. In the particular case where $F=P_k$, we write $\ind_k(G)$ instead of $\ind_{P_k}(G)$ for any $k\geq 3$.
\end{defn}

A subset $S\subseteq V$ is called a {\it complete} of $G$ if $G[S]$ is isomorphic to a complete graph. 
In particular, a complete that is maximal with respect to inclusion is called a \emph{clique} of $G$.
A graph $G$ is \emph{chordal} if every induced cycle in $G$ has length at most $3$. 
A \emph{simplicial vertex} is a vertex $v$ such that $N[v]$ is a clique. Every chordal graph has a simplicial vertex due to Dirac \cite{Dir}.

\subsection{Simplicial Complexes}
An \emph {abstract simplicial complex} $\Delta$ on a finite vertex set $V$ is a set of subsets of $V$, called \emph{faces}, satisfying the following properties:
\begin{enumerate}
\item $\{v\}\in \Delta$ for all $v\in V$.
\item If $F\in\Delta$ and $H\subseteq F$, then $H\in\Delta$.
\end{enumerate}

For a given a subset $U\subset V$, the complex $\Delta[U]:=\{\sigma\colon \sigma\in \Delta\;\textrm{,}\;\sigma\subseteq U\}$ is called the \emph{induced subcomplex} by $U$. 
If two simplicial complexes $\Delta$ and $\Delta'$ are isomorphic, we denote it by $\Delta\iso \Delta'$.
A simplicial complex $\Delta$ is called \emph{flag} if each of its minimal non-faces consists of two elements.

Let $\Delta$ be a simplicial complex. For a given face $\sigma$, the \emph{link} $\lk_\Delta(\sigma)$ and the \emph{deletion} $\dl_\Delta(\sigma)$ 
are defined respectively by $\lk_\Delta(\sigma)=\{\tau\in\Delta\colon\tau\cap \sigma=\emptyset\;\textrm{and}\;\tau\cup\sigma\in\Delta\}$ and 
$\dl_\Delta(\sigma)=\{\tau\in\Delta\colon\tau\cap\sigma=\emptyset\}$. For a vertex $x$ in $\Delta$, we abbreviate $\dl_{\Delta}(\{x\})$ and 
$\lk_\Delta(\{x\})$ to $\dl_{\Delta}(x)$ and $\lk_{\Delta}(x)$ respectively. 

The \emph{independence complex} of a graph $G=(V,E)$ is the simplicial complex on $V$ consisting independent sets of $G$ and is denoted by $\IE(G)$. 
The \emph{dominance complex} of a graph $G=(V,E)$ is the simplicial complex $\Dom(G)=\{\sigma\colon V\setminus\sigma\;\textrm{is a dominating set of}\;G\}$. 
Equivalently, the minimal non-faces of $\Dom(G)$ are the minimal elements of $\{N[x]\colon x\in V\}$.

Throughout this paper, $S^n$ will denote the $n$-dimensional sphere. If two topological spaces $X$ and $Y$ are homotopy equivalent, 
we denote it by $X \simeq Y$. The join and the wedge of two simplicial complexes $\Delta_0$ and $\Delta_1$ are denoted by 
$\Delta_0*\Delta_1$ and $\Delta_0\vee\Delta_1$, respectively. A simplicial complex $\Delta$ on $V$ is a cone with apex 
$v\in V$ if for every $\sigma\in\Delta$ we have $\sigma\cup \{v\}\in\Delta$. A well known fact is that if a simplicial 
complex $\Delta$ is a cone with apex $v$, then it is contractible. The suspension of a simplicial complex $\Delta$ and 
the cone over $\Delta$ will be denoted by $\Sigma \Delta$ and $Cone(\Delta)$, respectively.

The followings are well-known in Combinatorial Topology \cite{AH,Bj,Jo}. 
\begin{thm}\label{T1}
Let $\Delta$ be a simplicial complex. If $\lk_{\Delta}(x)$ is contractible in $\dl_{\Delta}(x)$ 
for some vertex $x$ of $\Delta$, then $\Delta\simeq \dl_{\Delta}(x)\vee\Sigma \lk_{\Delta}(x)$ holds.
\end{thm}

\begin{thm}\label{T3}Let $\Delta$ be a contractible simplicial complex and let $\Gamma$ 
be a nonempty subcomplex of $\Delta$. Then the quotient complex $\Delta/ \Gamma$ is homotopy equivalent to the suspension $\Sigma(\Gamma)$. 
\end{thm}
For spaces $X_0,X_1,A$ with $A\subseteq X_1$ and a map $f:A\rightarrow X_0$, the space $X_0$ with $X_1$ 
\emph{attached along A via f} is a quotient space obtained from the disjoint union of $X_0$ and $X_1$ 
by identifying each point $a\in A$ with its image $f(a)\in X_0$ and is denoted by $X_0\sqcup_fX_1$.

\begin{thm}\cite{AH}\label{T4} If $(X_1,A)$ is a CW pair and the two attaching maps $f,g:A\rightarrow X_0$ are homotopic, 
then $X_0\sqcup_fX_1\simeq X_0\sqcup_gX_1$.
\end{thm}
The following theorem is a special instance of the Bj\"{o}rner's generalized homotopy complementation formula.

\begin{thm}\cite{Bjo}\label{T2} Let $\Delta$ be a simplicial complex on $V$ and assume that there exists a subset 
$A\subseteq V$ such that $\dim(\Delta[A])=0$ and 
$\Delta[V\setminus A]$ is contractible, then we have the homotopy equivalence 
$\Delta\simeq \vee _{x\in A} \Sigma(\lk_\Delta(x))$.
\end{thm}

\subsection{Discrete Morse theory}
Discrete Morse theory was introduced by R. Forman in \cite{F}, and now it is counted as one of the most powerful techniques in topological combinatorics. 
The existence of a Morse matching on a simplicial complex $\Delta$ enables us 
to perform collapses for every matched pair in the matching, and therefore a new $CW$ 
complex is formed having the same homotopy type with the initial simplicial complex $\Delta$.

To every simplicial complex $\Delta$, one can associate a poset $P(\Delta)$ called the 
\emph{face poset} of $\Delta$, which is the set of faces of $\Delta$ ordered by inclusion.
Now consider the Hasse diagram of the face poset $P(\Delta)$ which is a directed graph with 
edges pointing down from large to small elements. A set $M$ of pairwise disjoint edges of 
this graph is called a \emph{matching} of $P(\Delta)$. So a matching $M$ corresponds exactly 
to a pairing of faces of $\Delta$ such that each face appears at most once. A matching $M$ is 
\emph{perfect} if it covers all elements of the face poset $P(\Delta)$.

When we have the Hasse diagram of the poset $P(\Delta)$ and a matching $M$, a modified Hasse 
diagram can be constructed by reversing the direction of the edges contained in $M$. 
A matching $M$ is said to be \emph{Morse}, if the modified Hasse diagram is acyclic. 
The main theorem of discrete Morse theory can now be stated as follows.

\begin{thm}\cite{Fo} Let $\Delta$ be a simplicial complex with a Morse matching $M$. Assume that for each $i\geq 0$, 
there are $c_j$ unmatched $i$-dimensional simplices. Then, $\Delta$ is homotopy equivalent to a $CW$ 
complex with exactly $c_j$ cells of each positive dimension $i$ and $c_0+1$ cells of dimension 0.
\end{thm}

\begin{cor}Let $\Delta$ be a simplicial complex with a Morse matching $M$ such that $c_j=0$ for all 
but one $i$. Then for this particular $i$, we have $\Delta\simeq \vee_{c_i}S^{i}$.
\end{cor}

The following lemma is due to Jonsson \cite{Jo} that allows us to combine acyclic matchings on 
families of subsets of a finite set to form a larger acyclic matching.

\begin{lem}\label{L1}Let $\Delta_0$ and $\Delta_1$ be disjoint families of subsets of a finite 
set such that $\tau\nsubseteq \sigma$ if $\sigma\in \Delta_0$ and $\tau\in \Delta_1 $. If $\ME_i$ 
is an acyclic matching on $\Delta_i$ for $i=0,1$, then $\ME_0\cup \ME_1$ is an acyclic matching on $\Delta_0\cup\Delta_1$.
\end{lem}
\subsection{Matching Trees For The Independence Complexes}
In this subsection we recall the definition of matching trees which were introduced by Bousquet-M\'elou, et al. in \cite{BLN} 
to construct a Morse matching on the independence complex $\IE(G)$ of a graph $G=(V,E)$. 

Let $\Sigma(A,B):=\{I\in \IE(G)\colon A\subseteq I\;\textrm{and}\;B\cap I=\emptyset\}$, where $A,B\subseteq V$ 
are two subsets such that $A\cap B=\emptyset$ and $N(A):=\cup_{a\in A}N(a)\subseteq B$.

The nodes of the matching tree represent sets of yet unmatched elements and they are of the form $\Sigma(A,B)$. The root of the matching tree is $I(\emptyset,\emptyset)=\IE(G)$, and other nodes are defined recursively as follows. If the node is the empty set, it is declared as a leaf. Otherwise, the node is of the form $\Sigma(A,B)$, which is a non-empty set. If $A\cup B=V$, this node also is declared as a leaf. Then the remaining nodes are of the form $\Sigma(A,B)$, with $A\cup B\neq V$. Pick a vertex $p$ in $V'=V\setminus (A\cup B)$ and proceed as follows:
\begin{itemize}
\item If $p$ has at most one neighbour in $V'$, define $\Delta(A,B,p)$ to be the subset of $\Sigma(A,B)$ formed of sets that do not intersect $N(p)$. 
So $\Delta(A,B,p)=\{I\in \IE(G)\colon A\subseteq I\;\textrm{and}\;B\cap I=I\cap N(p)=\emptyset\}$. Then  
$$M(A,B,p)=\{(I,I\cup \{p\})\colon I\in \Delta(A,B,p)\;\textrm{and}\;p\notin I\}$$ 
gives a perfect matching of $\Delta(A,B,p)$ and $p$ is called a \emph{pivot} of this matching. The node $\Sigma(A,B)$ has a unique child, namely the set 
$U=\Sigma(A,B)\setminus \Delta(A,B,p)$ of unmatched elements. If $p$ has no neighbour in $V'$, this set is empty. If $p$ has exactly one neighbour $v$ in $V'$, 
then $U=\Sigma(A\cup \{v\},B\cup N(v))$. The $3$-tuple $(A,B,p)$ is called a \emph{matching site} of the tree.
\item Otherwise, choose one neighbour $v$ of $p$ in $V'$. The node $\Sigma(A,B)$  has two children. 
The left child is $\Sigma(A,B\cup \{v\})$ and the right child is  $\Sigma(A\cup \{v\},B\cup N(v))$. Here, $(A,B,p)$ is called a \emph{splitting site} of the tree. 
\end{itemize}

\begin{thm}\cite{BLN} For any graph $G$ and any matching tree of $G$, the matching of $\IE(G)$ obtained by taking the union of all partial matchings 
$M(A,B,p)$ performed at the matching sites is Morse.
\end{thm}

\section{Some Reduction Techniques}\label{section:reduct}
In this section, we introduce some homotopy reduction techniques in greater generality, and provide some applications.

Let $\Delta$ be a simplicial complex on $V$. For a given subset $A\subseteq V$, we define, 
$$\DE(A):=\{H\in \Delta\colon H\cup A\notin \Delta\}.$$ 
In particular, if $A$ has only one element $p$, we simply write $\DE(p)$ instead of $\DE(\{p\})$, and set $\DE[v]:=\DE(v)\cup \{\{v\}\}$. 

\begin{thm}\label{T5}Let $\Delta$ be a simplicial complex. If $u$ and $v$ are vertices of $\Delta$ such that $\DE(u)\subseteq \DE(v)$, then $\Delta$ is homotopy equivalent to $\dl_\Delta(v)$.
\end{thm}
\begin{proof}Suppose that $u$ and $v$ are vertices of $\Delta$ satisfying $\DE(u)\subseteq \DE(v)$. However, this containment implies that $\lk_{\Delta}(v)=\Delta\setminus \DE[v]$ that
in turn forces $\lk_{\Delta}(v)$ to be a cone with apex $u$. The theorem now follows from Theorem \ref{T1}.
\end{proof}

\begin{exmp}Consider the simplicial complex $\Delta$ depicted in Figure \ref{Fig-1}. Note that we have $\DE(1)=\{\{3,4\}\}$ and $\DE(2)=\{\{4\},\{1,4\},\{3,4\}\}$ 
so that $\DE(1)\subseteq \DE(2)$ holds. We therefore conclude that $\Delta\simeq \dl_{\Delta}(2)$.
\begin{figure}[ht]
\begin{center}
\psfrag{G}{$\Delta$}\psfrag{H}{$\simeq$}
\psfrag{1}{$1$}\psfrag{2}{$2$}\psfrag{3}{$3$}\psfrag{4}{$4$}
\includegraphics[width=3in,height=1.5in]{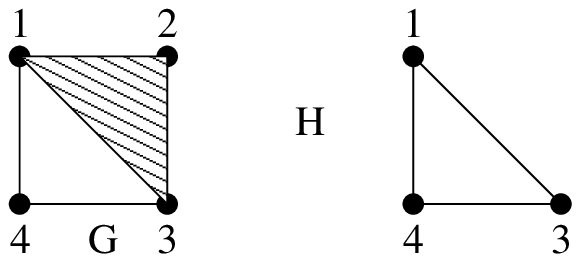}
\end{center}
\caption{}
\label{Fig-1}
\end{figure}
\end{exmp}
\begin{thm}\label{thm:split}
Let $\Delta$ be a simplicial complex. If $u$, $v$ are vertices of $\Delta$ such that $\DE[u]\subseteq \DE[v]$, then the following homotopy equivalence
\begin{equation*}
\Delta\simeq \dl_\Delta(v)\vee\Sigma \lk_\Delta(v)
\end{equation*}
holds.
\end{thm}
\begin{proof} We first prove that there is a cone $C$ such that $\lk_\Delta(v)\subseteq C\subseteq \dl_\Delta(v)$ so that $\lk_\Delta(v)$ is contractible in $\dl_\Delta(v)$. 
Assume that $A\in \lk_\Delta(v)$. It follows that $A\notin D[v]$. However, the assumption $\DE[u]\subseteq\DE[v]$ implies that $A\notin D[u]$ as well. Therefore, we have 
that $A\cup\{u\}$ is a face of $\Delta$ so that $C=\lk_\Delta(v)*u$ is the desired cone. Now for the second part, we need to show that the inclusion map 
$i\colon \lk_\Delta(v)\rightarrow \dl_\Delta(v)$ is null-homotopic. The map $H:\lk_\Delta(v)\times [0,1]\rightarrow \dl_\Delta(v)$ defined by $H(x,t)=(1-t)x+tu$ 
gives the desired homotopy equivalence between the maps $i$ and the constant map $c: \lk_\Delta(v)\rightarrow \dl_\Delta(v)$ with $c(x)=u$. 
The result now follows from Theorem \ref{T1}.
\end{proof}

We next consider the homotopical effect of removing a non-face from a simplicial complex, and compare the homotopy type of the resulting complex with that of 
the initial complex. For a simplicial complex $\Delta$ on $V$ and a minimal non-face $K$ of $\Delta$, we define
 $\FE^\Delta_K:=\{F\subseteq V\colon K\subseteq F\;\textrm{such that}\;\textrm{if}\;\ K\neq H\subsetneq F\;\textrm{then}\;\ H\in \Delta\}$.

\begin{thm}\label{T1*}Let $\Delta$ be a simplicial complex. If $K$ is a minimal non-face of $\Delta$ such that $\Delta[K]*\lk_{\Delta'}(K)$ is contractible 
in $\Delta$, then we have the following homotopy equivalence:
\begin{equation*}
\Delta'\simeq \Delta\vee\Sigma(S^{|K|-2}*\lk_{\Delta'}(K)),
\end{equation*}
where $\Delta'=\Delta\cup \{K\}\cup \KE$ and $\KE\subseteq \FE^\Delta_K$.
\end{thm}
\begin{proof} We write $\Delta'$ as a union of two subcomplexes $\Delta_0=\Delta$ and $\Delta_1=K*\lk_{\Delta'}(K)$. Note that $\Delta_0\cap \Delta_1=\Delta[K]*\lk_{\Delta'}(K)$ 
and $\Delta_1$ is contractible. Now let $i:\Delta_0\cap \Delta_1\rightarrow \Delta_0$ be the identity embedding and let $c:\Delta_0\cap \Delta_1\rightarrow \Delta_0$ be any 
constant map. Since $\Delta[K]*\lk_{\Delta'}(K)$ is contractible in $\Delta_0$, the maps $i$ and $c$ are homotopic. It then follows that
$\Delta_0\cup_i\Delta_1\simeq \Delta_0\cup_c\Delta_1$ by Theorem \ref{T4}. On the other hand, we have
$\Delta_0\cup_c\Delta_1\simeq \Delta_0\vee \Delta_1/(\Delta_0\cap \Delta_1)\simeq \Delta_0\vee \Sigma(\Delta_0\cap \Delta_1)$ by Theorem~\ref{T3}, 
since $\Delta_0\cap \Delta_1$ is contractible in $\Delta_0$. 
\end{proof}

\begin{rem}
We note that in the specific case of the independence complexes, Theorems~\ref{T5} and~\ref{thm:split} reduce to the fold lemma of~\cite{En} and the splitting result~\cite{A, MT}.
Similarly, Theorem~\ref{T1*} generalizes the following splitting result of Adamaszek in~\cite{A}:If $\IE(e\cup (G\setminus N[e])$ is contractible in $\IE(G)$, then there is a splitting 
$\IE(G\setminus e)\simeq \IE(G)\vee\Sigma^2 (\IE(G\setminus N[e]))$, where $e$ is an edge of $G$.
\end{rem}
\begin{rem}
Recall that for a simplicial complex $\Delta$, the removal of a pair of faces $\{\gamma,\tau\}$ is called an \emph{elementary collapse}, if $\tau$ is a unique 
maximal face containing $\gamma$ and $\dim(\tau)=\dim(\gamma)+1$. It is a well-known fact in simple-homotopy theory~\cite[Section 11.1]{Bj} that an elementary collapse preserves the homotopy type.
This result can be proven easily by Theorem~\ref{T1*}. Namely, since $\lk_\Delta(\gamma)=*$, it follows that $\Delta'\simeq \Delta$ by Theorem \ref{T1*}, 
where $\Delta=\Delta'\cup \{\gamma\}\cup\{\tau\}$.
\end{rem}

\begin{exmp}Consider the simplicial complex $\Delta_1$ depicted in Figure \ref{Fig-2}. We note that $\lk_{\Delta_1}(\{b,e\})$ is contractible so that by Theorem \ref{T1*}, 
that $\Delta_1$ is homotopy equivalent to the simplicial complex $\Delta_2$ depicted in Figure \ref{Fig-2}. Furthermore, since $\Delta_3[\{d,e\}]*\lk_{\Delta_2}(\{d,e\})$ 
is contractible in ${\Delta_3}$, we have $\Delta_2\simeq \Delta_3\vee \Sigma(S^0*\emptyset)$ by Theorem \ref{T1*} that implies the homotopy equivalence $\Delta_2\simeq S^1\vee S^1$. 
\end{exmp}

\begin{figure}[ht]
\begin{center}
\psfrag{H}{$\Delta_1$}\psfrag{G}{$\Delta_2$}\psfrag{K}{$\Delta_3$}
\psfrag{a}{$a$}\psfrag{b}{$b$}\psfrag{c}{$c$}\psfrag{d}{$d$}\psfrag{e}{$e$}
\includegraphics[width=2.7in,height=1in]{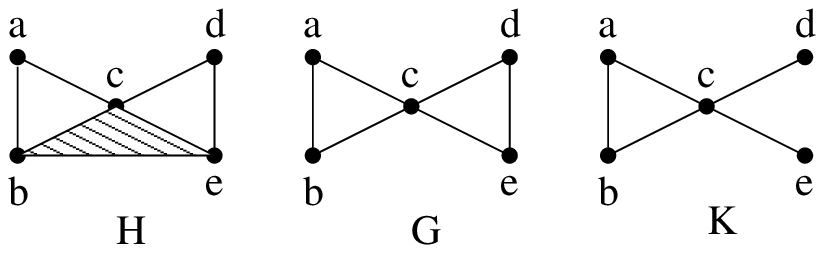}
\end{center}
\caption{}
\label{Fig-2}
\end{figure}

\begin{defn} Let $G=(V,E)$ be a graph and let $\FE$ be a family of graphs. Then \emph{devoid complex} $\DE(G;\FE)$ of the graph $G$ with 
respect to $\FE$ is a simplicial complex on $V$ whose faces are those subsets $S\subseteq V$ such that $G[S]$ is $F$-free for all $F\in\FE$.
\end{defn}
We particularly note that the independence complexes are examples of devoid complexes, since $\IE(G)\cong \DE(G;\{K_2\})$ for any (simple) graph $G$. 
We also abbreviate $\DE(G;\FE)$ to $\DE(G;F)$ whenever $\FE=\{F\}$. 

We next provide a splitting theorem for a particular devoid complex.

\begin{thm}\label{T6} Let $\DE(G;\FE')$ be a devoid complex of a graph $G=(V,E)$ with respect to $\FE'$, where $\FE'=\FE\cup \{C_2\}$, and $\FE$ is 
a family of graphs such that no member of the family $\FE$ is isomorphic to a graph that
contains an isolated vertex. For a vertex $u$ of $G$, if $G[N_G[u]]$ is a double complete graph (i.e. any two vertices induce a double-edge), 
then we have the following homotopy equivalence:
\begin{equation*}
\DE(G;\FE')\simeq\vee_{v\in N(u)}\Sigma(\DE(G;\FE')\setminus \DE[v])
\end{equation*}
\end{thm}
\begin{proof}
Since the set $A=N(u)$ satisfies $\dim(\DE(G;\FE')[A])=0$, and $\DE(G;\FE')[V\setminus A]$ is contractible, the result follows from Theorem \ref{T2}.
\end{proof}

As an application of Theorem~\ref{T5}, we next calculate the homotopy type of the devoid complex $\DE(P_n;P_k)$.
\begin{thm}\label{thm:paths} For the devoid complex $\DE(P_n;P_k)$, the following homotopy equivalence holds:
\begin{equation*}
\DE(P_n;P_k)\simeq
\begin{cases}
S^{tk-t-1}, & \text{if $n=(k+1)t$},\\
S^{tk-t+k-2}, & \text{if $n=(k+1)t+k$},\\
*, & \text{otherwise}\\
\end{cases}
\end{equation*}
where $n\geq k$.
\end{thm}
\begin{proof} We begin by choosing a different labeling of vertices of $P_n$ which is more suitable for our purposes. 
So, assume that $n=(k+1)t+d$, where $0\leq d\leq k$, and let the vertices of $P_n$ are labeled as $ij$ 
where $1\leq i\leq \lceil \frac{n-d}{k+1}\rceil$, $1\leq j\leq k+1$, and the vertices $ij$ and $rl$ form an edge if and only if $i=r$ and $|j-l|=1$ or $|i-r|=1$ and $j=k+1$, $l=1$.
\begin{enumerate}
\item Suppose that $d=0$, i.e. $n=(k+1)t$. Since $\DE(11)\subseteq \DE(1(k+1))$, we have $\DE(P_n;P_k)\simeq \DE(P_n\setminus 1(k+1);P_k)$. 
Since $\DE(21)\subseteq \DE(2(k+1))$ for the complex $\DE(P_n\setminus 1(k+1);P_k)$, we have $\DE(P_n\setminus 1(k+1);P_k)\simeq \DE(P_n\setminus \{1(k+1),2(k+1)\};P_k)$. 
If we continue in this way,  we conclude that $$\DE(P_n;P_k)\simeq \DE(P_n\setminus \{1(k+1),2(k+1),\dots,t(k+1)\};P_k).$$
The result now follows by noticing that the graph $(P_n\setminus \{1(k+1),2(k+1),\dots,t(k+1)\}$ is isomorphic to $t$ disjoint $k$-paths.
\item If $d\leq k-1$, then $\DE(P_n;P_k)$ is contractible, since when we apply the same procedure as in $(1)$, we obtain $(t+1)$ disjoint paths, and one of them has length $d$.
\item Now suppose that $d=k$. Applying the same procedure as above, the resulting graph is isomorphic to $(t+1)$ disjoint $k$-paths.
\end{enumerate}
\end{proof}

Our next task is to calculate the homotopy type of $\DE(F;P_3)$, where $F$ is a forest. However, we first need some technical results.

\begin{defn}Let $F$ be a forest and let $H$ be the subgraph of $F$ obtained by removing all leaves in $F$.  
Then a vertex $x$ in $F$ is called a \emph{saddle vertex} of $F$ if the degree of $x$ in $H$ is $1$ or $0$.
\end{defn}
\begin{lem}
Let $T$ be a tree with $|T|\geq 3$. Then there exists at least one saddle vertex $x$ of $T$.
\end{lem}
\begin{proof}
After removing all leaves of the tree $T$, we obtain a forest, and every forest has a vertex of degree $0$ or degree $1$.
\end{proof}

\begin{lem}\label{L2}Let $T$ be a tree with $|T|\geq 3$. Suppose that $T'=T\setminus N[x]$ and $T''=T\setminus (N[x]\cup N[w])$, 
where $x$ is a saddle vertex of $T$, and $w$ is the unique non-leaf neighbour of $x$ (if exists). Then we have the followings:
\begin{enumerate}
\item $\ind_3(T'')\leq \ind_3(T)-1$.
\item If $d_G(x)\geq 3$, then $\ind_3(T')=\ind_3(T)-1$.
\end{enumerate}
\end{lem}
\begin{proof}
Suppose that $T$ is a tree with $|T|\geq 3$ and let $T'$ and $T''$ be the subgraphs of $T$ such that $T'=T\setminus N[x]$ and $T''=T\setminus (N[x]\cup N[w])$, 
where $x$ is a saddle vertex of $T$. Moreover, let $w$ be the unique non-leaf neighbour of $x$ if it exists.
\begin{enumerate}
\item Let $a$ be a leaf neighbour of $x$ in $T$. Assume that $M''$ is a maximum induced $P_3$-matching of $T''$. Then $M''\cup \{P\}$ is an induced $P_3$-matching of $T$, 
where $P$ is the $3$-path with vertex set $\{a,x,w\}$ so that the inequality $\ind_3(T'')\leq \ind_3(T)-1$ holds.
\item Now assume that $d_G(x)\geq 3$. Let $a,u\in N(x)$, $b\in N(w)$ and $c\in N(b)$ be arbitrary vertices of $T$ such that $a,u\neq w$, $b\neq x$ and $c\neq w$. 
If $M'$ is a maximum induced $P_3$-matching of $T'$, then $M'\cup \{P\}$ is an induced $P_3$-matching of $T$, 
where $P$ is the $3$-path with vertices $a,x,u$; hence, the inequality $\ind_3(T')\leq \ind_3(T)-1$ holds.

Now suppose that $M$ is a maximum induced $P_3$-matching of $T$. If $M$ contains a $3$-path from the subgraph $T[N[x]\setminus \{w\}]$, 
then $M$ can not contain a $3$-path with a vertex $w$. We therefore have $\ind_3(T)=|M|=1+\ind_3(T')$. 
If $M$ contains a $3$-path $P$ with vertex set $\{a,x,w\}$, then $M$ can not contain any $3$-path with a vertex in $\{a,x,w\}$ or any neighbours of them. 
We obtain the subgraph $T''$ of $T$ by removing these vertices. Notice that $T''$ is a subgraph of $T'$. 
We therefore have $\ind_3(T)=|M|=1+\ind_3(T'')\leq 1+\ind_3(T')$. If $M$ contains a $3$-path $P$ with vertices $x,w,b$, 
then $M$ can not contain any $3$-path with a vertex in $\{x,w,b\}$ or any neighbours of them. 
Let $T_1$ be the subgraph of $T$ obtained by removing these vertices. We note in this case that $T_1$ is a subgraph of $T'$. 
Therefore, we have $\ind_3(T)=|M|=1+\ind_3(T_1)\leq 1+\ind_3(T')$. If $M$ contains a $3$-path $P$ with vertex set $\{w,b,c\}$, 
then $M$ can not contain any $3$-path with a vertex in $\{w,b,c\}$ or any neighbours of them. 
Let $T_2$ be the subgraph of $T$ obtained by removing these vertices. It is clear that $T_2$ contains all leaf neighbours of $x$ as isolated vertices. 
Since these vertices have no contribution to the number $\ind_3(T_2)$, we remove them so that the remaining graph $T_2^*$ is a subgraph of $T'$. 
Therefore, we have $\ind_3(T)=|M|=1+\ind_3(T_2)=1+\ind_3(T_2^*)\leq 1+\ind_3(T')$ which proves the inequality $\ind_3(T)\leq 1+\ind_3(T')$.
\end{enumerate}
\end{proof}

\begin{thm}\label{thm:forest} Let $F$ be a forest. Then $\DE(F;P_3)$ is contractible or is homotopy equivalent to a wedge of spheres of dimension at most $2\ind_3(F)-1$.
\end{thm}
\begin{proof} It is enough to prove this for a tree $T$. If $x$ is a saddle vertex of $T$, the complex $\dl_{\DE(T,\{P_3\})}(x)$ is contractible. We therefore have 
$\DE(T;P_3)\simeq \Sigma \lk_{\DE(T;P_3)}(x)$. Note that $\lk_{\DE(T;P_3)}(x)\cong \DE(G_x;\{P_3,C_2\})$. Here, the graph
$G_x$ is the graph obtained from $T$ by placing an extra edge (resp. a double-edge) between any two vertices $a$ and $b$
if $ab\in E$ (resp.  $ab\notin E$), whenever the set $\{a,b,x\}$ induces a $3$-path, and then deleting the vertex $x$. 
Observe that if $u\in N_T(x)$ and $u$ is a leaf, then $N_{G_x}[u]$ induces a double complete graph in $G_x$. 
Now we apply Theorem \ref{T6} to conclude the homotopy equivalence: $$\DE(T;P_3)\simeq\vee_{v\in N_{G_x}(u)}\Sigma^2(\DE(G_x;\{P_3,C_2\})\setminus \DE[v]).$$ 
If $T\cong K_{1,n}$, then $\DE(T;P_3)\simeq\vee_{v\in N_{G_x}(u)}\Sigma^2S^{-1}$. Thus $\DE(T;P_3)$ is homotopy equivalent to a wedge of spheres of dimension $1$. 
On the other hand, if $T\ncong K_{1,n}$, then there exists a vertex $w$ in $T$ which is the unique non-leaf neighbour of $x$. 
If $v\neq w$, then we have $\DE(G_x;\{P_3,C_2\})\setminus \DE[v]\simeq\DE(T';P_3)$, where $T'=T\setminus N[x]$, whereas if $v=w$, 
then $\DE(G_x;\{P_3,C_2\})\setminus\DE[v]\simeq \DE(T'';P_3)$, where $T''=F\setminus (N_T[x]\cup N_T[w])$.

Now, if $d_T(x)=2$, we then have  $$\DE(T;P_3)\simeq\vee_{v\in N_{G_x}(u)}\Sigma^2(\DE(G_x;\{P_3,C_2\})\setminus\DE[v])=\Sigma^{2}\DE(T'';P_3).$$
It follows that $\ind_3(T'')\leq \ind_3(T)-1$ by Lemma \ref{L2}. Therefore, $\DE(T'';P_3)$ is contractible or is homotopy equivalent to a wedge of spheres of dimension 
at most $2(\ind_3(T)-1)-1)$ by the induction; hence, the complex $\DE(T;P_3)$ is contractible or is homotopy equivalent to a wedge of spheres of dimension at most $2\ind_3(T)-1$ as claimed.

If $d_T(x)>2$, then $\DE(T;P_3)\simeq\vee_{|N_{G_x}(u)|-1}\Sigma^{2}\DE(T';P_3)\vee\Sigma^{2}\DE(T'';P_3)$. We analyze every component of this wedge keeping 
in mind that $\ind_3(T')=\ind_3(T)-1$ and $\ind_3(T'')\leq \ind_3(T)-1$ by Lemma \ref{L2}. If a component of this wedge is not contractible 
then $\DE(T';P_3)$ and $\DE(T'';P_3)$ is homotopy equivalent to a wedge of spheres of dimension at most $2\ind_3(T')-1$ and $2\ind_3(T'')-1$, respectively. Applying the suspensions, we obtain a wedge of spheres of dimension at most $2\ind_3(T')-1+2=2\ind_3(T')+1=2(\ind_3(T)-1)+1=2\ind_3(T)-1$ or $2\ind_3(T'')-1+2\leq 2(\ind_3(T)-1)+1\leq 2\ind_3(T)-1$. This completes the proof.
\end{proof}

\section{Matching Trees For a Simplicial Complex}
This section is devoted to the construction of a Morse matching through matching trees for an arbitrary simplicial complex. 
Our departure in this direction begins by generalizing the notion of matching trees which is due to Bousquet-M\'elou, et al. in \cite{BLN} for the case of the independence complexes.

Main motivation for this section comes from the following idea. Let $\Delta$ be a simplicial complex on a vertex set $V$. 
We pick an element $p$ of $V$, and define $\Delta_p=\{I\in \Delta\colon\;\textrm{if}\;H\in\DE(p)\;\textrm{then}\;H\nsubseteq I\}$. Then, we notice that the set of pairs $(I,I\cup \{p\})$ forms a perfect matching of $\Delta_p$; hence, a matching of $\Delta$. We call $p$ the \emph{pivot} of this matching. The unmatched elements of $\Delta$ are exactly the faces of $\Delta$ which contain at least one element from the family $\DE(p)$. We now choose another pivot $p'$ to match some elements of $\Delta\setminus \Delta_p$, and continue the same procedure.

Now let us describe the construction of a matching tree to facilitate the study. The nodes of the matching tree represent sets of yet unmatched elements, and they are of the form
\begin{equation*}
\Sigma(A,\mathcal{B}):=\{I\in \Delta\colon A\subseteq I\;\textrm{and}\;\emptyset\neq B\nsubseteq I\;\textrm{for all}\;B\in \mathcal{B}\},
\end{equation*}
where $A\subseteq V$, and for all $D\in \DE(A)$, there exists some $B\in \mathcal{B}$ such that $B\subseteq D$, where $\mathcal{B}$ is a family of subsets 
of $V$ for which $A\cap V(\mathcal{B}^*)=\emptyset$ and $\mathcal{B}^*$ is the subfamily of $\mathcal{B}$ consisting of all one point sets of $\mathcal{B}$ with ground set $V(\mathcal{B}^*)$. The root of the matching tree is $I(\emptyset,\{\emptyset\})=\Delta$, and other nodes are defined recursively as follows. If the node is the empty set, it is declared as a leaf. Otherwise, the node is of the form $\Sigma(A,\mathcal{B})$, which is a non-empty set. If $A\cup V(\mathcal{B}^*)=V$, then $\Sigma(A,\mathcal{B})= \{A\}$ is a node with an unmatched element of cardinality $|A|$ in which case we also declare this node a leaf. The remaining nodes are of the form $\Sigma(A,\mathcal{B})$ with $A\cup V(\mathcal{B}^*)\neq V$. 
Furthermore, if $\Sigma(A,\mathcal{B})$ is a child of $\Sigma(A',\mathcal{B}')$, we then define $\DE_{up}^{\mathcal{B}'}(A)$ to be the set consisting of 
the minimal elements of the family $\DE(A)\cup \{B'\setminus A\colon B'\in\mathcal{B}'\}$. Now we pick a vertex $p$ in $V'=V\setminus (A\cup V(\mathcal{B}^*))$ and proceed as follows:
\begin{itemize}
\item If there exists at most one element $\gamma\in \DE_{up}^{\mathcal{B}}(p)$ such that $B\nsubseteq \gamma$ for all $B\in \mathcal{B}$, we define $\Delta(A,\mathcal{B},p)$ 
to be the subset of $\Sigma(A,\mathcal{B})$ formed of sets that do not contain any element from 
$\DE_{up}^\mathcal{B}(p)$, that is, 
$$\Delta(A,\mathcal{B},p)=\{I\in \Delta\colon A\subseteq I, \emptyset\neq B\nsubseteq I\;\textrm{for all}\;B\in \mathcal{B}\;\textrm{and if}\;H\in\DE_{up}^\mathcal{B}(p)\;
\textrm{then}\;H\nsubseteq I\}.$$ 

Then, $M(A,\mathcal{B},p)=\{(I,I\cup \{p\})\colon I\in \Delta(A,\mathcal{B},p)\;\textrm{and}\;p\notin I\}$ gives a perfect matching of 
$\Delta(A,\mathcal{B},p)$ and $p$ is called a pivot of this matching.

We should note that if $I\in \Delta(A,\mathcal{B},p)$, then $I\cup \{p\}$ is in $\Delta(A,\mathcal{B},p)$: We clearly have $A\subseteq I\cup\{p\}$. 
Moreover, if $H\in\DE_{up}^{\mathcal{B}}(p)$, then $p\notin H$ so that $H\nsubseteq I\cup\{p\}$. 
We claim that $B\nsubseteq I\cup \{p\}$ for any $B\in \mathcal{B}$. To verify this, assume otherwise that there exists $B\in \mathcal{B}$ such that $B\subseteq I\cup \{p\}$.
However, this forces $p\in B$, since $B\nsubseteq I$. It follows that there exists some $C\in \DE_{up}^{\mathcal{B}}(p)$ such that $C\subseteq B\setminus \{p\}$; hence,
$C\nsubseteq I\cup \{p\}$ which implies that $B\nsubseteq I\cup \{p\}$, a contradiction. Therefore, we have $B\nsubseteq I\cup \{p\}$ as claimed.

Now, associate a unique child to the node $\Sigma(A,\mathcal{B})$, namely the set $U=\Sigma(A,\mathcal{B})\setminus \Delta(A,\mathcal{B},p)$ of unmatched elements. 
Note that if for any $\gamma\in\DE_{up}^{\mathcal{B}}(p)$, there exists some $B\in\mathcal{B}$ such that $B\subseteq \gamma$, then $U=\emptyset$. 
If there exists exactly one element $\gamma\in \DE_{up}^{\mathcal{B}}(p)$ such that $B\nsubseteq \gamma$ for all $B\in \mathcal{B}$, 
then $U=\Sigma(A\cup \gamma,\mathcal{B}\cup \DE_{up}^{\mathcal{B}}(\gamma))$. The $3$-tuple $(A,\mathcal{B},p)$ is called a \emph{matching site} of the tree. 
We label the new edge by the pivot $p$.

\item If there exists at least two elements $\gamma,\gamma'\in \DE_{up}^{\mathcal{B}}(p)$ satisfying $B\nsubseteq \gamma$ and $B\nsubseteq \gamma'$ 
whenever $B\in \mathcal{B}$, then the set $U$ is not of the form $\Sigma(A',\mathcal{B}')$. This is because some of the unmatched elements $I$ contain $\gamma$, 
some others do not. However, we eventually want all nodes to be of the form $\Sigma(A',\mathcal{B}')$. Such a problem can be resolved by first splitting the original set 
$\Sigma(A,\mathcal{B})$ into two disjoint subsets of the form $\Sigma(A',\mathcal{B}')$. By using, say $\gamma$, 
we can write $\Sigma(A,\mathcal{B})= \Sigma(A,\mathcal{B}\cup \{\gamma\})\uplus \Sigma(A\cup\gamma,\mathcal{B}\cup\DE_{up}^{\mathcal{B}}(\gamma))$ and then examine 
each subset separately. So, the node $\Sigma(A,\mathcal{B})$  has two children. 
The left child is  $\Sigma(A,\mathcal{B}\cup \{\gamma\})$ and the right child is $\Sigma(A\cup\gamma,\mathcal{B}\cup \DE_{up}^{\mathcal{B}}(\gamma))$. 
The 3-tuple $(A,\mathcal{B},\gamma)$ is called a \emph{splitting site} of the tree. We label the two new edges by $\gamma$.
\end{itemize}
\begin{exmp}Consider the simplicial complex $\Delta$ depicted in Figure \ref{Fig-3}. 
\end{exmp}
\begin{figure}[ht]
\begin{center}
\psfrag{G}{$\Delta$}\psfrag{1}{$1$}\psfrag{2}{$2$}\psfrag{3}{$3$}
\psfrag{4}{$4$}\psfrag{5}{$5$}
\includegraphics[width=1in,height=1in]{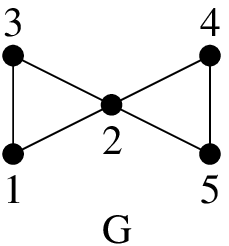}
\end{center}
\caption{}
\label{Fig-3}
\end{figure}
We will construct a matching tree of $\Delta$, and to simplify the notation, we write $\overline{\DE}(p)$ instead of $\DE_{up}^{\mathcal{B}}(p)$. 
We choose the vertex $1$ as a pivot. The sets $\{2,3\}$ and $\{4\}$ belong to the family $\overline{\DE}(1)$. So, the set  $\overline{\DE}(1)$ 
has at least two elements. Let us choose one element from $\overline{\DE}(1)$, say $\{2,3\}$. Then $\Sigma(\emptyset,\{\emptyset\})$ has two children, namely they 
are $\Sigma(\emptyset,\{\{2,3\}\})$ and $\Sigma(\{2,3\},\{\{1\},\{4\},\{5\}\})$. We illustrate it in Figure \ref{Fig-4} and 
label the edge with the chosen set $\{2,3\}$. We have that $\Sigma(\{2,3\},\{\{1\},\{4\},\{5\}\})= \{\{2,3\}\}$. We now choose 
the vertex $2$ as a pivot for the node $\Sigma(\emptyset,\{\{2,3\}\})$. Then the sets $\{3\}$ and $\{4,5\}$ belong to the family $\overline{\DE}(2)$. 
We next consider the element $\{3\}$ from the family $\overline{\DE}(2)$. Then $\Sigma(\emptyset,\{\{2,3\}\})$ has two children, namely they are 
$\Sigma(\emptyset,\{\{3\}\})$ and $\Sigma(\{3\},\{\{2\},\{4\},\{5\}\})$. Now let us first find the children of $\Sigma(\{3\},\{\{2\},\{4\},\{5\}\})$.
For that purpose, we choose the vertex $1$ as a pivot, and it is clearly the only possibility. Now, $\mathcal{B}=\{\{2\},\{4\},\{5\}\})$, and for any $\gamma\in \overline{\DE}(1)$, there exists $B\in \mathcal{B}$ such that $B\subseteq \gamma$. 
Thus, the child of $\Sigma(\{3\},\{\{2\},\{4\},\{5\}\})$ is just the empty 
family. Now, consider the vertex $2$ as a pivot for the node $\Sigma(\emptyset,\{\{3\}\})$. The node $\Sigma(\emptyset,\{\{3\}\})$ has only 
one child, namely $\Sigma(\{4,5\},\{\{1\},\{2\},\{3\}\})$, since $\overline{\DE}(2)$ has exactly one element which does not contain $\{3\}$. 
Thus the unmatched elements of $\Delta$ are $\{2,3\}$ and $\{4,5\}$; hence, we conclude that $\Delta\simeq S^1\vee S^1$. 

Our matching tree gives the Morse matching $\{(\emptyset,2),(1,12),(4,24),(5,25),(3,13)\}$.
     
\begin{figure}[ht]
\begin{center}
\psfrag{G}{$\Delta$}\psfrag{1}{$1$}\psfrag{2}{$2$}\psfrag{3}{$3$}
\psfrag{4}{$4$}\psfrag{5}{$5$}\psfrag{a}{$\Sigma(\emptyset,\{\emptyset\})$}
\psfrag{b}{$\Sigma(\emptyset,\{\{2,3\}\})$}
\psfrag{c}{$\Sigma(\{2,3\},\{\{1\},\{4\},\{5\}\})$}
\psfrag{d}{$\Sigma(\emptyset,\{\{3\}\})$}
\psfrag{e}{$\Sigma(\{3\},\{\{2\},\{4\},\{5\}\})$}
\psfrag{f}{$\Sigma(\{4,5\},\{\{1\},\{2\},\{3\}\})$}\psfrag{g}{$\emptyset$}
\psfrag{h}{$\{2,3\}$}\psfrag{k}{$\{3\}$}\psfrag{m}{$1$}\psfrag{n}{$2$}
\includegraphics[width=3in,height=3in]{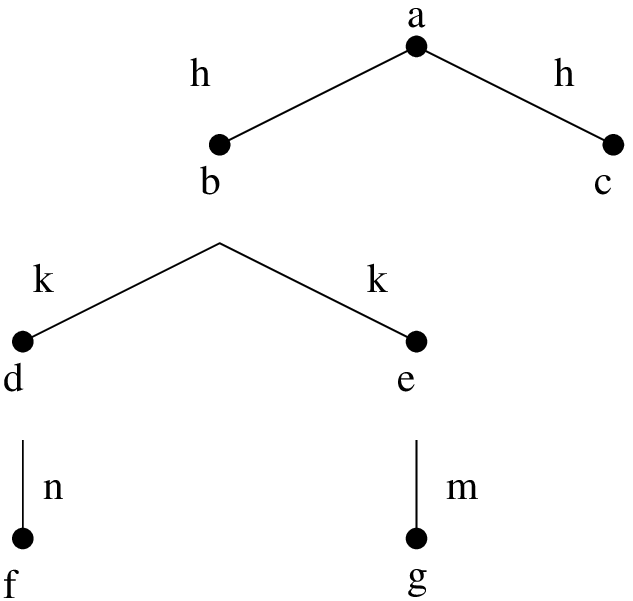}
\end{center}
\caption{}
\label{Fig-4}
\end{figure}

\begin{lem}\label{L3}Every matching tree satisfies the following properties:
\begin{enumerate}
\item For every matching site $(A,\mathcal{B},p)$, the matching $M(A,\mathcal{B},p)$ is a Morse matching of $\Delta(A,\mathcal{B},p)$.
\item Let $(A,\mathcal{B},p)$ be a matching site with a non-empty child $\Sigma(A\cup \gamma,\mathcal{B}\cup \DE_{up}^{\mathcal{B}}(\gamma))$. 
Let $I\in \Delta(A,\mathcal{B},p)$ and $J\in \Sigma(A\cup \gamma,\mathcal{B}\cup \DE_{up}^{\mathcal{B}}(\gamma))$ be given. Then $J\nsubseteq I$.
\item Let $(A,\mathcal{B},\gamma)$ be a splitting site, and let $I\in \Sigma(A,\mathcal{B}\cup\{\gamma\})$ and 
$J\in \Sigma(A\cup \gamma,\mathcal{B}\cup \DE_{up}^{\mathcal{B}}(\gamma))$ be given. Then $J\nsubseteq I$.
\end{enumerate}
\end{lem}
\begin{proof}
\begin{enumerate}
\item We know that $M(A,\mathcal{B},p)$ is a perfect matching of $\Delta(A,\mathcal{B},p)$. Hence, we need only to verify that the modified Hasse diagram of the 
face poset $\Delta(A,\mathcal{B},p)$ is acyclic. Now, consider a directed edge in the modified Hasse diagram. The up edges join two elements of the 
form $I\setminus \{p\}$ and $I$, so they correspond to adding the vertex $p$. The down edges correspond to deleting a vertex different from $p$. 
This cannot lead to a directed cycle; hence, the claim holds.
\item Assume that $J\subseteq I$. If $J\in \Sigma(A\cup \gamma,\mathcal{B}\cup \DE_{up}^{\mathcal{B}}(\gamma))$, 
then $\gamma\subseteq J$. Moreover, if $I\in \Delta(A,\mathcal{B},p)$, then $\gamma\nsubseteq I$, since $\gamma\in \DE_{up}^{\mathcal{B}}(p)$. However, it follows that
$\gamma\subseteq J\subseteq I$, which in turn implies $\gamma\subseteq I$, a contradiction.
\item If $I\in \Sigma(A,\mathcal{B}\cup \{\gamma\})$ and $J\in \Sigma(A\cup \gamma,\mathcal{B}\cup \DE_{up}^{\mathcal{B}}(\gamma))$, 
then we have $\gamma\nsubseteq I$ and $\gamma\subseteq J$ so that $J\nsubseteq I$ as claimed.
\end{enumerate}
\end{proof}
\begin{thm}
For any simplicial complex $\Delta$ and any matching tree of $\Delta$, the matching of $\Delta$ obtained by taking the union of all partial matchings $M(A,\mathcal{B},p)$ performed at the matching sites is Morse.
\end{thm}
\begin{proof}
We prove this by backward induction from the leaves to the root, that is, we show that the union of the partial matchings performed at 
the descendants of a node $\Gamma$ (including $\Gamma$ itself) of the matching tree is a Morse matching of $\Gamma$. We denote this 
matching by $UM(\Gamma)$ (the union of matchings). The empty matching is performed at the leaves of the tree and it is therefore Morse. 
Now consider a non-leaf node $\Sigma(A,\mathcal{B})$ of the tree. Suppose that $(A,\mathcal{B},p)$ is a matching site implying $M(A,\mathcal{B},p)$ 
to be a Morse matching by Lemma \ref{L3}. Since it is a matching site, it has a unique child which is either empty or of the form 
$\Sigma(A\cup\gamma,\mathcal{B}\cup \DE_{up}^{\mathcal{B}}(\gamma))$. If the child is empty, then $UM(\Gamma)=M(A,\mathcal{B},p)$ 
which is Morse, and we are done. If the child is $\Gamma'=\Sigma(A\cup\gamma,\mathcal{B}\cup \DE_{up}^{\mathcal{B}}(\gamma))$, 
then $UM(\Gamma'$) is Morse by induction hypothesis. If we write $\Gamma=\Delta(A,\mathcal{B}, p)\uplus \Gamma'$, then 
the Morse matchings on $\Delta(A,\mathcal{B}, p)$ and $\Gamma'$ are $M(A,\mathcal{B},p)$ and $UM(\Gamma')$ respectively. 
By applying Lemma~\ref{L1} and~\ref{L3}, we conclude that $UM(\Gamma)$ is Morse.

Now suppose that  $(A,\mathcal{B},\gamma)$ is a splitting site. If the children are $\Sigma(A,\mathcal{B}\cup \{\gamma\})$ and
$\Sigma(A\cup\gamma,\mathcal{B}\cup \DE_{up}^{\mathcal{B}}(\gamma))$, then we can write 
$\Gamma=\Sigma(A,\mathcal{B}\cup \{\gamma\})\uplus \Sigma(A\cup\gamma,\mathcal{B}\cup \DE_{up}^{\mathcal{B}}(\gamma))$. 
By the induction hypothesis, we have the Morse matchings $UM(\Sigma(A,\mathcal{B}\cup \{\gamma\})$ and 
$UM(\Sigma(A\cup\gamma,\mathcal{B}\cup \DE_{up}^{\mathcal{B}}(\gamma))$ on $\Sigma(A,\mathcal{B}\cup \{\gamma\})$ and $\Sigma(A\cup\gamma,\mathcal{B}\cup \DE_{up}^{\mathcal{B}}(\gamma))$ 
respectively. Once again, by Lemma~\ref{L1} and~\ref{L3}, the union $UM(\Gamma)$ of $UM(\Sigma(A,\mathcal{B}\cup \{\gamma\})$ and 
$UM(\Sigma(A\cup\gamma,\mathcal{B}\cup \DE_{up}^{\mathcal{B}}(\gamma))$ 
is Morse. This completes the induction. 
\end{proof}

\section{Some Applications Of Matching Trees on Simplicial Complexes}
In this section, we compute the homotopy type of some devoid complexes as an application of the techniques introduced in the previous section.
In more detail, we calculate the homotopy type of $\DE(C_n;P_k)$ for which we first characterize the unmatched elements of $\DE(P_n;P_k)$  
in order to find those of $\DE(C_n;P_k)$. Also note that the case where $k=2$ corresponds to the independence complexes
of cycles whose homotopy types have already been computed by Kozlov~\cite{Koz,Ko}. Furthermore, 
we calculate the homotopy types of dominance complexes of chordal graphs that complements the work of
Marietti and Testa~\cite{MT,MaT}. 

We assume that the graph $P_n$ has vertex set $[1,n]=\{1,2,\dots,n\}$ and two vertices $i$ and $j$ form an edge if and only if $|i-j|=1$.

\begin{lem} For the devoid complex $\DE(P_n;P_k)$, where $n\geq k$ we have the following properties:
\begin{enumerate}
\item if $n=(k+1)t$ then there exists a unique unmatched element of cardinality $tk-t$,
\item if $n=(k+1)t+k$ then there is a unique unmatched element of cardinality $tk-t+k-1$,
\item otherwise, there is no unmatched element in $\DE(P_n,\{P_k\})$.
\end{enumerate}
\end{lem}
\begin{proof}
We choose the vertex $1$ as our first pivot. Then the family $\overline{\DE}(1)$ has only one element, namely $\{2,3,4,\dots,k\}$. 
If $n=k=(k+1)t+k=(k+1)0+k$, then $\overline{\DE}(\{2,3,4,\dots,k\})=\{\{1\}\}$. Note that $\Sigma(\emptyset,\{\emptyset\})$ has a unique child 
$\Sigma(\{2,3,4,\dots,k\},\{\{1\}\})=\{\{2,3,4,\dots,k\}\}$. Therefore it has a unique unmatched element of cardinality $k-1=k0-0+k-1$. 
This verifies the induction base. Now consider the case where $n>k$. Then $\overline{\DE}(\{2,3,4,\dots,k\})=\{\{1\},\{k+1\}\}$ so that 
$\Sigma(\emptyset,\{\emptyset\})$ has a unique child, namely $\Sigma(\{2,3,4,\dots,k\},\{\{1\},\{k+1\}\})$. The graph 
obtained by deleting $\{2,3,4,\dots,k\}\cup \{1,k+1\}$ is $P_{n-(k+1)}$; hence, the unmatched elements of $\DE(P_n;P_k)$ 
can be obtained by adding the vertices ${2,3,\dots,k}$ to the unmatched elements of $\DE(P_{n-(k+1)};P_k)$. Suppose 
that $n=(k+1)t$. Then $n-(k+1)=(k+1)t-k-1=(k+1)(t-1)$. By the induction hypothesis, we conclude that $\DE(P_{n-(k+1)};P_k)$ 
has a unique unmatched element of cardinality $(t-1)k-(t-1)$. Therefore, $\DE(P_n;P_k)$ has a unique unmatched element of cardinality  $(t-1)k-(t-1)+k-1=tk-t$. 
Assume now that $n=(k+1)t+k$. Then $n-(k+1)=(k+1)t+k-k-1=(k+1)t-1=(k+1)(t-1)+k$. By induction hypothesis $\DE(P_{n-(k+1)};P_k)$ 
has a unique unmatched element of cardinality $(t-1)k+k-(t-1)-1=kt-t$. So, the complex $\DE(P_n;P_k)$ has a unique unmatched element 
of cardinality  $kt-t+k-1$. Finally, suppose that $n=(k+1)t+d$, where $1\leq d\leq k-1$, which implies that $n-(k+1)=(k+1)t+d-(k+1)=(k+1)t+d-k-1$. However,
we then have that the latter is equivalent to $(k+1)t+d$ by $\bmod({k+1})$. More precisely, $(k+1)t+d-k-1\equiv (k+1)t+d-k-1+k+1\equiv (k+1)t+d\pmod{k+1}$. 
By the induction hypothesis, $\DE(P_{n-(k+1)};P_k)$ has no unmatched element. Thus, $\DE(P_n;P_k)$ has no unmatched element. This completes the proof.
\end{proof}

Suppose that the vertex set of $C_n$ is $\{1,2,\dots,n\}$ such that $i$ and $j$ form an edge if and only 
if $|i-j|=1 \pmod {n}$ where $i,j\in \{1,2,\dots,n\}$. 
In order to simplify the notation, we write $I_i^j=[i,j]\subseteq [1,n]=\{1,2,\dots,n\}$ for the set $\{i,i+1,\dots,j\}$.

\begin{thm}\label{thm1} For the devoid complex $\DE(C_n;P_k)$, the following homotopy equivalence holds:
\begin{equation*}
\DE(C_n;P_k)\simeq
\begin{cases}
\bigvee^k S^{t(k-1)-1}, & \text{if $n=(k+1)t$}\\
S^{t(k-1)-1}, & \text{if $n=(k+1)t+1$}\\
S^{t(k-1)+d-2}, & \text{if $n=(k+1)t+d$}\\
S^{t(k-1)+k-2}, & \text{if $n=(k+1)t+k$}\\
\end{cases},
\end{equation*}
where $2\leq d\leq k-1$.
\end{thm}
\begin{proof}
We again choose the vertex $1$ as our first pivot. 
Then $I(\emptyset,\{\emptyset\})=\DE(C_n;P_k)$ has two children, since the family $\overline{\DE}(1)$ has more than one element. 
If we choose the element $I_2^k=[2,k]$ from the family $\overline{\DE}(1)$, the children are
\begin{align*}
A_1=\Sigma(\emptyset,\{I_2^k\})
\end{align*}
and
\begin{align*}
B_1=\Sigma(I_2^k,\{\{1\},\{k+1\}\}).
\end{align*}
Now we can choose $2$ as a pivot element for $A_1$. There exists at least two elements in $\overline{\DE}(2)$ which does not 
contain $I_2^k$ for $A_1$, and we choose $I_3^k$ from that family. Then the children of $A_1$ are given by
\begin{align*}
A_2=\Sigma(\emptyset,\{I_2^k,I_3^k\})=\Sigma(\emptyset,\{I_3^k\})
\end{align*}
and
\begin{align*}
B_2=\Sigma(I_3^k,\{I_2^k,\{2\},I_{k+1}^{k+2}\})=\Sigma(I_3^k,\{\{2\},I_{k+1}^{k+2}\}).
\end{align*}
Similarly, choose now $3$ as a pivot for $A_2$, and consider the element $I_4^k$ from the family $\overline{\DE}(3)$ for $A_2$. The children of  $A_2$ are
\begin{align*}
A_3=\Sigma(\emptyset,\{I_2^k,I_3^k,I_4^k\})=\Sigma(\emptyset,\{I_4^k\})
\end{align*}
and
\begin{align*}
B_3=\Sigma(I_4^k,\{I_2^k,I_3^k,\{3\},I_{k+1}^{k+3}\})=\Sigma(I_4^k,\{\{3\},I_{k+1}^{k+3}\}).
\end{align*}
We continue in this manner by branching each new $A_i$ using the vertex $i+1$ as a pivot and denoting the children of $A_i$ by $A_{i+1}$ and $B_{i+1}$. 
We stop this procedure when we reach the pivot $k-1$. We write the last two branchings in details. 
When we have that $k-2$ is our pivot, $A_{k-3}$ has two children, namely
\begin{align*}
A_{k-2}=\Sigma(\emptyset,\{I_2^k,I_3^k,I_4^k,\dots,I_{k-1}^k\})=\Sigma(\emptyset,\{I_{k-1}^k\})
\end{align*}
and
\begin{align*}
B_{k-2}=\Sigma(I_{k-1}^k,\{I_2^k,I_3^k,\dots,I_{k-2}^k,\{k-2\},I_{k+1}^{2k-2}\})=\Sigma(I_{k-1}^k,\{\{k-2\},I_{k+1}^{2k-2}\}).
\end{align*}
We take $k-1$ as a pivot for $A_{k-2}$ so that the children of $A_{k-2}$ are 
\begin{align*}
A_{k-1}=\Sigma(\emptyset,\{I_2^k,I_3^k,I_4^k,\dots,I_{k-1}^k,\{k\}\})=\Sigma(\emptyset,\{\{k\}\})
\end{align*}
and
\begin{align*}
B_{k-1}=\Sigma(\{k\},\{I_2^k,I_3^k,\dots,I_{k-1}^k,\{k-1\},I_{k+1}^{2k-1}\})=\Sigma(\{k\},\{\{k-1\},I_{k+1}^{2k-1}\}).
\end{align*}
Now note that $A_{k-1}=\Sigma(\emptyset,\{\{k\}\})=\DE(P_{n-1};P_k)$. Furthermore, the graph obtained by deleting $I_2^k\cup \{1,k+1\}$ is $P_{n-(k+1)}$. 
Thus the elements of $B_1$ can be obtained by adding ${2,3,\dots,k}$ to the unmatched elements of $\DE(P_{n-(k+1)};P_k)$. So, 
we only need to examine the sets $B_i$ for all $2\leq i\leq k-1$. 

A typical element of these sets can be written of the form
\begin{align*}
B_{k-l}=\Sigma(I_{k-l+1}^k,\{\{k-l\},I_{k+1}^{2k-l}\}).
\end{align*}
Note that $B_1$ corresponds to the case where $k-l=1$ so that we assume that $k-l\neq 1$, and by using appropriate pivots, we next describe the children of $B_{k-l}$.

We choose $\{k-l-1\}$ as a pivot. Since no element in $B_{k-l}$ contains $\{k-l\}$, it follows that $\overline{\DE}(k-l-1)$ has only one element, namely $I_{n-l}^{k-l-2}$. 
At this point, we should also note that $|I_{n-l}^{k-l-2}|=k-1$. Thus there exists a unique child of  $B_{k-l}$, and its child is $C_{k-l-1}=\Sigma(A,\mathcal{B})$, where
\begin{align*}
A=I_{k-l+1}^k\cup I_{n-l}^{k-l-2}
\end{align*}
and
\begin{align*}
\mathcal{B}=\{\{k-l\},I_{k+1}^{2k-l},\{k-l-1\},\{n-l-1\}\}.
\end{align*}

We now consider $\{n-l-2\}$ as a pivot. Since no element in  $C_{k-l-1}$ contains $\{n-l-1\}$, there exists a unique child $C_{n-l-2}=\Sigma(A,\mathcal{B})$ of $C_{k-l-1}$, where
\begin{align*}
A=I_{k-l+1}^k\cup I_{n-l}^{k-l-2}\cup I_{n-l-k-1}^{n-l-3},
\end{align*}
\begin{align*}
\mathcal{B}=\{\{k-l\},I_{k+1}^{2k-l},\{k-l-1\},\{n-l-1\},\{n-l-2\},\{n-l-k-2\}\}. 
\end{align*}

If we choose $n-l-k-3$ as the pivot, $C_{n-l-2}$ has a unique child $C_{n-l-k-3}=\Sigma(A,\mathcal{B})$, where
\begin{align*}
A&=I_{k-l+1}^k\cup I_{n-l}^{k-l-2}\cup I_{n-l-k-1}^{n-l-3}\cup I_{n-2k-l-2}^{n-l-k-4} \quad \textrm{and}\\
\mathcal{B}=\{\{k-l\},&I_{k+1}^{2k-l},\{k-l-1\},\{n-l-1\},\{n-l-2\},\\
&\{n-l-k-2\},\{n-l-k-3\},\{n-2k-l-3\}\}. 
\end{align*}
We continue in the same fashion until we reach the element $3k-l+d+1$ as our pivot in which case the unique child is $C_{3k-l+d+1}=\Sigma(A,\mathcal{B})$, where
\begin{align*}
A&=I_{k-l+1}^k\cup I_{n-l}^{k-l-2}\cup I_{n-l-k-1}^{n-l-3}\cup I_{n-2k-l-2}^{n-l-k-4}\cup \dots\cup I_{2k-l+d+2}^{3k-l+d} \quad \textrm{and}\\
\mathcal{B}=\{\{k-l\},&I_{k+1}^{2k-l},\{k-l-1\},\{n-l-1\},\{n-l-2\},\{n-l-k-2\},\\
&\{n-l-k-3\},\{n-2k-l-3\},\dots,\{3k-l+d+1\},\{2k-l+d+1\}\}.
\end{align*}
Now assume that $n=(k+1)t+d$, where $d\geq 0$. We can choose $2k-l+d$ as a pivot so that the unique child is $C_{2k-l+d}=\Sigma(A,\mathcal{B})$, where
\begin{align*}
A&=I_{k-l+1}^k\cup I_{n-l}^{k-l-2}\cup I_{n-l-k-1}^{n-l-3}\cup I_{n-2k-l-2}^{n-l-k-4}\cup\dots\cup I_{2k-l+d+2}^{3k-l+d}\cup I_{k-l+d+1}^{2k-l+d-1},\quad \textrm{and}\\ 
\mathcal{B}=\{&\{k-l\},I_{k+1}^{2k-l},\{k-l-1\},\{n-l-1\},\{n-l-2\}, \{n-l-k-2\},\{n-l-k-3\},\\
&\{n-2k-l-3\},\dots,\{3k-l+d+1\},\{2k-l+d+1\},\{2k-l+d\},\{k-l+d\}\}. 
\end{align*}

We will examine this case in detail:

{\bf Case.1:} Suppose that $d=0$. Then the node above is $C_{2k-l}=\Sigma(A,\mathcal{B})$, where
\begin{align*}
A&=I_{k-l+1}^k\cup I_{n-l}^{k-l-2}\cup I_{n-l-k-1}^{n-l-3}\cup I_{n-2k-l-2}^{n-l-k-4}\cup\dots\cup I_{2k-l+2}^{3k-l}\cup I_{k-l+1}^{2k-l-1}\quad \textrm{and}\\ 
\mathcal{B}=\{&\{k-l\},I_{k+1}^{2k-l},\{k-l-1\},\{n-l-1\},\{n-l-2\},\{n-l-k-2\},\\
&\{n-l-k-3\},\{n-2k-l-3\},\dots,\{3k-l+1\},\{2k-l+1\},\{2k-l\}\}. 
\end{align*}

Note that there is no more vertices that we can use as a pivot, since $A\cup V(\mathcal{B}^*)=V(C_n)$. Moreover, we have 
$I_{k-l+1}^k\subset I_{k-l+1}^{2k-l-1}$ so that $I_{n-l}^{k-l-2}\uplus I_{n-l-k-1}^{n-l-3}\uplus I_{n-2k-l-2}^{n-l-k-4}\uplus\dots\uplus I_{2k-l+2}^{3k-l}\uplus I_{k-l+1}^{2k-l-1}$ 
is the unique unmatched element of cardinality $t(k-1)$.

{\bf Case.2:} Assume that $1\leq d\leq k-1$.
\begin{itemize}
\item Suppose first that $d\leq l$. We return back to the pivot $2k-l+d$. If there is an element in $\overline{\DE}(2k-l+d)$, 
it must be of the form $[k-l+d+1,2k-l+d-1]$, since no set in $C_{3k-l+d+1}$ contains $\{2k-l+d+1\}$. Observe that $k-l+d+1\leq k+1$ due to the assumption that $d\leq l$. 
Therefore the set $[k-l+d+1,2k-l+d-1]$ must contain the set $[k+1,2k-l]$ which is impossible, since no set in $C_{3k-l+d+1}$ contains $[k+1,2k-l]$. 
Thus $C_{3k-l+d+1}$ has no children implying that there is no unmatched element.

\item Assume that $d=l+1$. We then have $C_{2k-l+d}=C_{2k+1}=\Sigma(A,\mathcal{B})$, where
\begin{align*}
A&=I_{k-l+1}^k\cup I_{n-l}^{k-l-2}\cup I_{n-l-k-1}^{n-l-3}\cup I_{n-2k-l-2}^{n-l-k-4}\cup\dots\cup I_{2k+3}^{3k+1}\cup I_{k+2}^{2k}\quad \textrm{and}\\ 
\mathcal{B}&=\{\{k-l\},I_{k+1}^{2k-l},\{k-l-1\},\{n-l-1\},\{n-l-2\},\{n-l-k-2\},\\
&\{n-l-k-3\},\{n-2k-l-3\},\dots,\{3k+2\},\{2k+2\},\{2k+1\},\{k+1\}\}. 
\end{align*}
Therefore, $I_{k-l+1}^k\uplus I_{n-l}^{k-l-2}\uplus I_{n-l-k-1}^{n-l-3}\uplus I_{n-2k-l-2}^{n-l-k-4}\uplus\dots\uplus I_{2k+3}^{3k+1}\uplus I_{k+2}^{2k}$
is the unique unmatched element of cardinality $t(k-1)+l=t(k-1)+d-1$.

\item Now suppose that $l+2\leq d\leq k$ and choose $k-l+d-1$ as a pivot. 
Any set in $\overline{\DE}(k-l+d-1)$ must be of the form $I_{d-l}^{k-l+d-2}$. Since $d-l\leq k-l$, this set must contain the vertex $k-l$. However 
this is impossible, since no set in $C_{2k-l+d}$ contain $k-l$. Thus there exists no unmatched element.
\end{itemize}

We next summarize the results we have so far in order to complete the proof:

\begin{enumerate}
\item $n=(k+1)t$. We recall that
\begin{align*}
A_{k-1}=\Sigma(\emptyset,\{\{k\}\})=\DE(P_{n-1};P_k).
\end{align*}
Since $n-1=(k+1)t-1=(k+1)(t-1)+k$, we conclude that $A_{k-1}$ has a unique unmatched element of cardinality $t(k-1)$. 
Also recall that the elements of $B_1$ can be obtained by adding ${2,3,\dots,k}$ to the unmatched elements 
of $\DE(P_{n-(k+1)};P_k)$. Since $n-(k+1)=(k+1)t-k-1=(k+1)(t-1)$, it follows that  $B_1$ has a unique unmatched 
element of cardinality $(t-1)k-(t-1)+k-1=t(k-1)$. For each $B_i$, there is a unique unmatched element of 
cardinality $t(k-1)$, where $2\leq i\leq k-1$. We therefore have $k$ unmatched elements of cardinality $t(k-1)$ in this case.
\item $n=(k+1)t+1$. Only $A_{k-1}$ has an unmatched element and its cardinality is $t(k-1)$.
\item $n=(k+1)t+d$ with $2\leq d\leq k-1$. The only unmatched element comes from a $B_i$ in the case $d=l+1$, and its cardinality is $t(k-1)+d-1$.
\item $n=(k+1)t+k$. Then there exists a unique unmatched element coming from $B_1$ of cardinality $t(k-1)+k-1$.
\end{enumerate}
\end{proof}

Our next task is to compute the homotopy types of dominance complexes of chordal graphs as we promised. 

\begin{lem} Let $G$ be a (simple) graph with a vertex $u$ such that the set $N[u]$ induces a clique in $G$. 
Then $\ve(G')=\ve(G)-d_G(u)$, where $G'=G-N[u]$.
\end{lem}
\begin{proof} 
Assume that $C$ is a minimum covering set for $G$. Then $C$ must contain a set $A$ having exactly $d_G(u)$ elements from $N[u]$, 
since otherwise there would be some edges in $G[N[u]]$ that is not covered. It follows that $C\setminus A$ covers $G'$. 
We therefore have $\ve(G')\leq \ve(G)-d_G(u)$. Now, let $C'$ be a minimum covering set for $G'$. Clearly $C'\cup N(u)$ 
is a covering set for $G$ so that $\ve(G)\leq \ve(G')+d_G(u)$. 
\end{proof}
\begin{thm}Let $G$ be a chordal graph. Then the dominance complex $\Dom(G)$ of $G$ is homotopy equivalent to a sphere of dimension $\ve(G)-1$.
\end{thm}
\begin{proof} We show by an induction that there exists a unique unmatched element in $\Dom(G)$ of cardinality $\ve(G)$. 
Since $G$ is a chordal graph, there exists a vertex $u$ such that $N[u]$ is a clique. We choose $u$ as a pivot. Then the 
family $\overline{\DE}(u)$ has only one element,  namely $N(u)$. Therefore, $\Sigma(\emptyset,\{\emptyset\})$ has a 
unique child $\Sigma(A,\mathcal{B})$, where $A=N(u)$ and $\mathcal{B}$ consists of the sets of minimal elements of 
$$\mathcal{B}'=\{N[x]\cap (V\setminus N[u])\colon x\;\textrm{is adjacent to one of the vertices in}\;N(u)\}\cup\{\{u\}\}.$$ 
Now, if we write $G'=G-N[u]$, then we note that the elements in $\Sigma(A,\mathcal{B})$ can be obtained 
by adding $N(u)$ to the unmatched elements of $\Dom(G')$. So, $\Dom(G')$ has a unique unmatched element of dimension 
$\ve(G')$ by the induction; hence, the complex $\Dom(G)$ has a unique unmatched element of dimension 
$\ve(G')+d_G(u)=\ve(G)-d_G(u)+d_G(u)=\ve(G)$. 
\end{proof}
\begin{lem}\label{T7}
Let $T$ be a tree with a saddle vertex $x$. If the vertex $a$ is the unique non-leaf neighbour of $x$, 
then $\m(T)=\m(T^*)+1$, where $T^*=T\setminus (N[x]\setminus \{a\})$.
\end{lem}
\begin{proof}
Assume that $x$ is a saddle vertex of $T$, and let $a$ be the unique non-leaf neighbour of $x$. 
Suppose further that $M^*$ is a maximum matching of $T^*$. Let $u$ be a leaf neighbour of $x$. Clearly, the set $M^*\cup\{(u,x)\}$ 
is a matching of $T$ so that $\m(T)\geq \m(T^*)+1$. Let $M$ be a maximum matching of $T$. 
If $M$ contains an edge from the subgraph $G[N[x]\setminus\{a\}]$, say the edge $(u,x)$, then $M\setminus \{(u,x)\}$ 
is a matching of $T^*$; hence, $\m(T^*)\geq |M|-1=\m(T)-1$. If $M$ contains the edge $(x,a)$, 
then $M\setminus \{(x,a)\}$ is a matching of $T'$, where $T'$ is a subgraph of $T^*$ obtained by deleting $N[x]$ from $T$ . 
Thus we have $\m(T^*)\geq\m(T')\geq |M|-1=\m(T)-1$. Finally, note that all the other possibilities are covered by one of these two cases. 
\end{proof}
\begin{thm}\cite{MT,MaT}
The dominance complex $\Dom(F)$ of a forest $F$ is homotopy equivalent to a sphere of dimension $\m(F)-1$.
\end{thm}
\begin{proof} 
It is enough to verify this for a tree $T$. Suppose that $x$ is a saddle vertex of $T$, and let 
the vertex $u$ in $T$ be a leaf neighbour of $x$. We choose $u$ as a pivot. Then the family $\overline{\DE}(u)$ has only 
one element, that is, $\overline{\DE}(u)=\{\{x\}\}$. Therefore, $\Sigma(\emptyset,\{\emptyset\})$ has a unique child $\Sigma(A,\mathcal{B})$, 
where $A=\{x\}$ and $\mathcal{B}$ consists of the sets of minimal elements of 
$$\mathcal{B}'=\{N[y]\cap (V\setminus N[u])\colon y\;\textrm{is adjacent to}\;x\}\cup\{\{u\}\}.$$ 
Now let $T^*$ 
be the graph obtained by deleting $N[x]\setminus \{a\}$, where $a$ is the unique non leaf neighbour of $x$. 
Then the elements in $\Sigma(A,\mathcal{B})$ can be obtained by adding $\{x\}$ to the unmatched 
elements of $\Dom(T^*)$. Since the complex $\Dom(T^*)$ has a unique unmatched element of dimension $\m(T^*)$ by the induction, 
we conclude that $\Dom(T)$ has a unique unmatched element of dimension $\m(T^*)+1$, which is equal to $\m(T)$ by Lemma~\ref{T7}. 
\end{proof} 

\section*{Acknowledgement}
I would like to thank Professor Yusuf Civan for introducing me the devoid simplicial complexes and for his invaluable comments and generous encouragement.


\end{document}